\documentclass[12pt]{amsart}
  
\input xy
\xyoption{all}
  
\usepackage{amssymb}
\usepackage{amsmath}
\usepackage{amscd}
\usepackage{latexsym}
\begin{document}
 
\newtheorem{lemma}{Lemma}[section]
\newtheorem{prop}[lemma]{Proposition}
\newtheorem{cor}[lemma]{Corollary}
  
\newtheorem{thm}[lemma]{Theorem}
\newtheorem{Mthm}[lemma]{Main Theorem}
\newtheorem{con}{Conjecture}
\newtheorem{claim}{Claim}

\theoremstyle{definition}
  
\newtheorem{rem}[lemma]{Remark}
\newtheorem{rems}[lemma]{Remarks}
\newtheorem{defi}[lemma]{Definition}
\newtheorem{ex}[lemma]{Example}
                                                                                
\newcommand{\C}{\mathbb C}
\newcommand{\R}{\mathbb R}
\newcommand{\Q}{\mathbb Q}
\newcommand{\Z}{\mathbb Z}
\newcommand{\N}{\mathbb N}

\newcommand{\cg}{{\mathbb C}G}
\newcommand{\eog}{\ell^1G}
\newcommand{\etg}{\ell^2G}
\newcommand{\vNG}{{\mathcal N}\! G}
\newcommand{\ZNG}{Z({\mathcal N}\! G)}
\newcommand{\ag}{{\mathcal A}G}
\newcommand{\crg}{C^*_rG}
\newcommand{\cmg}{C^*G}
\newcommand{\EG}{\underline{E}G}
\newcommand{\EGamma}{\underline{E}\Gamma_g}
\newcommand{\EGam}{\underline{E}\Gamma}
\newcommand{\ctr}{\operatorname{ctr}}
\newcommand{\Tg}{{\mathcal T}_g}
\newcommand{\PSL}{\operatorname{PSL}(2,\R)}
\newcommand{\Hom}{\operatorname{Homeo}}
\newcommand{\Dif}{\operatorname{Diffeo}}
\newcommand{\Hoe}{\operatorname{Hoequ}}
\newcommand{\Out}{\operatorname{Out}}
\newcommand{\Rep}{\operatorname{Rep}}
\newcommand{\Aut}{\operatorname{Aut}}
\newcommand{\In}{\operatorname{int}}
\newcommand{\CI}{\operatorname{Coind_H^G}}
\newcommand{\CIo}{\operatorname{Coind_{G_0}^G}}
\newcommand{\TR}{\operatorname{Tran}}
%%%%%%%%%%%%%%%%%%%%%%%%%%%%%%%%%%%%%%%%%%%%%%%%%
\title[On transfer in bounded cohomology]
{On transfer in bounded cohomology}
%%%%%%%%%%%%%%%%%%%%%%%%%%%%%%%%%%%%%%%%%%%%%%%                                          
\author{Indira Chatterji}\thanks{Partially supported by NSF grant No. 0644613 and ANR grant JC-318197}
\author{and Guido Mislin}
%%%%%%%%%%%%%%%%%%%%%%%%%%%%%%%%%%%%%%%%%%%%%%%      
\address{Mapmo Universit\'e D'Orl\'eans, Orl\'eans}
\email{indira.chatterji@univ-orleans.fr}
\address{Mathematics Department, ETH Z\"urich}
\email{mislin@math.ethz.ch}
\curraddr{Mathematics Department, Ohio-State University}
\email{mislin@math.ohio-state.edu}

\begin{abstract}%\footnotemark
We define a transfer map in the setting of
bounded cohomology with certain metric $G$-module
coefficients. As an application, we extend a
theorem on the comparison map from Borel-bounded
to Borel cohomology (cf. \cite{CMPS}),
to cover the case of Lie groups with finitely many
connected components.

\end{abstract}

%\footnotetext{Version of April 6, 2009}}

\maketitle

%%%%%%%%%%%%%%%%%%%%%%%%%%%%%%%%%%%%%%%%%%%%%%%%%%%
\section{Introduction}
%%%%%%%%%%%%%%%%%%%%%%%%%%%%%%%%%%%%%%%%%%%%%%%%%%%
The groups $G$
we are working with are always assumed to be separable,
locally compact, 
with a topology given by a complete metric. Borel cohomology
groups $H^*_B(G,A)$ for $G$ with coefficients in a
polish $G$-module $A$ where introduced by Moore
\cite{Moore}. In case $A$ is equipped with a 
metric, one can define the Borel-bounded cohomology groups
$H^*_{Bb}(G,A)$ and study the comparison map from Borel-bounded
to Borel cohomology. The case of a connected Lie group $G$
and $A=\Z$ with trivial $G$-action was dealt with in \cite{CMPS}.
The purpose of this note is to define a transfer map in the
$H^*_{B}$ and $H^*_{Bb}$ setting, with respect to a closed subgroup
$H<G$ of finite index. We will make use
of our transfer map to extend some results proved
in \cite{CMPS} for connected Lie groups, to cover the case
of groups with finitely many connected components ({\sl{virtually
connected}} groups).

We thank Ian Leary for helpful remarks.

\section{The Borel (bounded) setup}\label{basic}

Through this note, $G$ will denote a separable, locally compact topological group,
with a topology given by a complete metric.

\begin{defi}
We denote by $P(G)$ the category of $G$-modules
which are separable, completely metrizable with an isometric
$G$-action ({\sl Polish} $G$-modules); morphisms are continuous $G$-maps. 
\end{defi}
When dealing with bounded cohomology, we need to fix
a metric on our coefficient modules. Let $M\in P(G)$. Let $m_1$ and
$m_2$ be two complete metrics on $M$ (compatible
with the given topology) with respect to which the $G$-action
is isometric. We say that $m_1$ and $m_2$ are {\sl bounded equivalent}
($b$-equivalent for short), if they define the same bounded
subsets in $M$. 

\begin{defi} We denote by $mP(G)$ the category of $G$-modules
$M$ in $P(G)$, equipped with a preferred $b$-equivalence class $[m]$
of metrics. Morphisms in $mP(G)$ are continuous $G$-maps which
map bounded subsets to bounded subsets.
\end{defi}

If $H<G$ is a closed normal subgroup of finite index
and $M$ is a finitely generated $\Z G$-module with trivial
$H$-action, then the word metric $m_S$ on the underlying
abelian group of $M$, with respect to a finite $G$-stable
generating set $S<M$ defines a $b$-equivalence class
$[m_S]$ of complete metrics on $M$, independent of
the particular $S$ chosen. In this way we 
obtain a fully faithful functor
$${\operatorname{fgMod}(G/H)}\to mP(G)\,,$$
with ${\operatorname{fgMod}(G/H)}$ denoting the
category of finitely generated $G/H$-modules. 

There is an obvious forgetful functor
$mP(G)\to P(G)$. In this way, we may view
an $M\in {\operatorname{fgMod}(G/H)}$ as an object
in $mP(G)$, respectively $P(G)$.

For $M\in P(G)$, a map $f: G^{n+1}\to M$ is called {\sl Borel}, if it is Borel measurable with respect to the $\sigma$-algebras associated to the topological spaces $G^{n+1}$ and $M$ respectively; in case $M$ is in $mP(G)$,
$f$ is called {\sl bounded}, if its range is a bounded subset of $M$; $f$ is called {\sl $G$-equivariant}, if $f(xg_0,\cdots,xg_n)
= xf(g_0,\cdots,g_n)$. We write
$$C^n_{B}(G,M)={\rm{map}}_{B}^G(G^{n+1},M)\hbox{ and }C^n_{Bb}(G,M)={\rm{map}}_{Bb}^G(G^{n+1},M)$$ 
for the abelian group of $G$-equivariant Borel (respectively Borel bounded) maps
$G^{n+1}\to M$. With the usual differentials, this defines
cochain complexes $C^*_{B}(G,M)$ and $C^*_{Bb}(G,M)$ whose respective cohomologies,
$$H^*_{B}(G,M)\hbox{ and }H^*_{Bb}(G,M)$$
is the {\sl Borel} (respectively {\sl Borel bounded}) cohomology of $G$ with coefficients
in $M$ (for Borel cohomology with
coefficients in $P(G)$, see Moore \cite{Moore}; Borel bounded
cohomology with Banach $G$-module coefficients has been dealt with
in \cite{Monod}).

%When the settings work for both Borel and Borel bounded theories, we shall use the subscripts $B(b)$.

\begin{rem}
Let $G$ denote a virtually connected Lie group with
connected component $G_0$.
For $M\in P(G)$ with underlying topology of $M$ discrete, the
$G$-action on $M$ factors through $G/G_0$ so that $M$ is a
$G/G_0$-module.
Note also that, because $G$ is virtually connected $\pi_1(BG)=\pi_0(G)=G/G_0$ is a finite group. According to
Wigner \cite[Theorem 4]{Wigner}, see also Moore \cite[Section 7, (1)]{Moore}, one has for $M$ 
discrete and countable
a natural isomorphism
$$H^*_B(G,M)\cong H^*(BG,M)$$
where $H^*(BG,M)$ denotes the singular cohomology of $BG$
with local coefficients in the $\pi_1(BG)$-module $M$. Similarly, for
$V$ a continuous finite-dimensional $\R$-representation of $G$, one has
a natural isomorphism
$$H^*_B(G,V)\cong H^*_c(G,V)\,,$$
with $H^*_c(G,V)$ denoting the continuous cohomology of $G$
\cite[Section 7, (2)]{Moore}.
\end{rem}

Let $H<G$ be a closed subgroup of finite index and
let $M\in P(H)$.
We write $\pi:G\to G/H$ for the projection and
we fix a section $\sigma: G/H\to G$ satisfying
$\sigma(1\cdot H)=1\in G$. There is a continuous retraction 
$\rho: G\to H$
given by $\rho(g)=g\cdot \sigma(\pi(g^{-1}))$. 
We define the coinduced module $\CI M\in P(G)$ as follows: 
the underlying abelian group consist
of all continuous functions $f:G\to M$ such that $f(gh)=h^{-1}f(g)$
for all $(g,h)\in G\times H$. This defines a coinduction functor
$$\CI: P(H)\to P(G)\,.$$
If $M\in mP(G)$, that is, $M$ is equipped with $d_M$ 
a compatible preferred metric, we observe that
the range of $d_M(f(?),0):G\to \R$ is compact, because
$d_M(f(gh),0)=d_M(f(g),0)$ for $h\in H$. We define a metric 
on $\CI{M}$, using the sup-metric $||f||=
\operatorname{sup}_G(d_M(f(g),0))$. The
isometric $G$-module structure on $\CI{M}$ is 
given by $(x\cdot f)(g)=f(x^{-1}g)$. Hence the functor $\CI$ defines a coinduction functor
$$\CI:mP(H)\to mP(G)\,.$$
It also follows that one has a natural isomorphism
$$(\CI{M})^G\stackrel{\cong}\to M^H\,, f\mapsto f(1)\,.$$
The following partial {\sl Eckmann-Shapiro}
Lemma holds in this situation.
\begin{lemma}\label{CI} Let $G$ be a separable, locally compact topological group,
with a topology given by a complete metric and let $H<G$ a closed subgroup of finite index. Let $M\in mP(H)$.
There is a natural diagram
$$\xymatrix{H^*_{Bb}(G,\CI{M})\ar[r]^{\hspace{5mm}\Theta^*_{b}}
\ar[d]& H^*_{Bb}(H,M)\ar[d]\\
H^*_{B}(G,\CI{M})\ar[r]^{\hspace{5mm}\Theta^*}&H^*_B(H,{M})}$$
and both maps $\Theta^*$ and $\Theta^*_b$ are onto.
\end{lemma}
\begin{proof} The map of cochain complexes:
$$\Theta: C^*_{B}(G,\CI{M})\to C^*_{B}(H,M)\,,$$
given on $n$-cochains by $\Theta(f)(h_0,\cdots,h_n)=f(h_0,\cdots,h_n)(1)$, restricts to a map $\Theta_b$ of bounded cochains. Moreover, the map
$$\Lambda: C^*_{B}(H,M)\to C^*_{B}(G,\CI{M})$$
given by $\Lambda(f)(g_0,\cdots,g_n)(g)=f(\rho(g^{-1}g_0),\cdots,\rho(g^{-1}g_n))$ on $n$-cochains, satisfies
\begin{eqnarray*}(\Theta\circ\Lambda)(f)(h_0,\cdots,h_n)&=&
\Lambda(f)(h_o,\cdots,h_n)(1)\\
&=&f(\rho(h_0),\cdots,\rho(h_n))=
f(h_0,\cdots,h_n)\,,\end{eqnarray*}
so that $\Theta\circ\Lambda=Id\,$ and therefore the induced map $\Theta^*$ and $\Theta^*_b$ are split surjections. 
\end{proof}
Recall that the {\sl radical}
$\sqrt{G}$ of a (not necessarily connected) Lie group $G$
is its maximal, connected, normal, solvable subgroup. A Lie
groups is called {\sl linear}, if it admits a faithful
representation $G\to GL(n,\R)$ for some $n$.
\begin{cor}\label{3implies1}
Let $G$ be a virtually connected Lie group such that
for every finitely generated $G/G_0$-module $M$ the forgetful
map $H^*_{Bb}(G,M)\to H^*_B(G,M)$ is surjective. Then the radical $\sqrt{G}$
of $G$ is linear.
\end{cor}
\begin{proof}
Since $\CIo{\Z}$ is
finitely generated as an abelian group, with the
connected group $G_0$ acting trivially, it is a
finitely generated $G/G_0$-module. We conclude from
our assumption that 
$$H^*_{Bb}(G,\CIo{\Z})\to H^*_B(G,\CIo{\Z})$$
is surjective. From Lemma \ref{CI} with $H=G_0$ and $M=\Z$ we conclude that
$$H^*_{Bb}(G_0,\Z)\to H^*_B(G_0,\Z)$$
is surjective too. Because $G_0$ is connected, we can apply
Theorem 1.1 of \cite{CMPS} to conclude that $\sqrt{G_0}$ is linear,
which completes the proof of the corollary, since $\sqrt{G_0}=\sqrt{G}$.
\end{proof}

\section{Construction of the transfer map}\label{transfer}
Let $G$ be a separable, locally compact topological group,
with a topology given by a complete metric and let $H<G$ a closed subgroup of finite index.
Let $M\in P(G)$. The map
$$\kappa: \CI{M}\to M,\quad f\mapsto \sum_{gH\in G/H}g\cdot f(g)$$
is well-defined, because for $h\in H$, $h^{-1}f(g)=f(gh)$ so that
$gh\cdot f(gh)=g\cdot f(g)$. Moreover, $\kappa$ is a $G$-equivariant map since for $x\in G$,
\begin{eqnarray*}\kappa(x\cdot f)&=& \sum_{G/H}g\cdot (xf)(g)=\sum_{G/H}g\cdot f(x^{-1}g)\\
&=&
x(\sum_{G/H}x^{-1}g\cdot f(x^{-1}g))=x\cdot \kappa(f).\end{eqnarray*} 

We further note that, when working with a fixed
compatible metric on $M$, then $||\kappa(f)||\le [G:H]||f||$. Hence $\kappa$ maps bounded sets of $\CI{M}$ to bounded sets of $M$ and induces
$$\kappa^*:H^*_{Bb}(G,\CI{M})\to H^*_{Bb}(G,M)\,.$$
\begin{defi}
We define the {\sl transfer map} in Borel cohomology for $M\in P(G)$ by
$$\TR: = \kappa^*\circ \Theta^*: H^*_{B}(H,M)\to H^*_{B}(G,\CI{M})\to H^*_{B}(G,M)\,.$$
If $M\in mP(G)$, one similarly defines the {\sl transfer map} in Borel bounded cohomology by
$$\TR: = \kappa^*\circ \Theta^*_b: H^*_{Bb}(H,M)\to H^*_{Bb}(G,\CI{M})\to H^*_{Bb}(G,M)\,.$$
\end{defi}
\begin{thm}Let $G$ be a separable, locally compact topological group,
with a topology given by a complete metric and let $H<G$ a closed subgroup of finite index. Let $M\in P(G)$ and denote by
$\rm{Res}: H^*_B(G,M)\to H^*_B(H,M)$ the restriction map. Then the composite map
$$\TR\circ\rm{Res}: H^*_B(G,M)\to H^*_B(G,M)$$
is the multiplication by the index $[G:H]$. Moreover, if $M\in mP(G)$,
there is a natural commutative diagram
$$\xymatrix{
H^*_{Bb}(G,M)\ar[r]^{\rm{Res}}\ar[d]&
H^*_{Bb}(G_0,M)\ar[r]^{\TR}\ar[d]&
 H^*_{Bb}(G,M)\ar[d]\\
  H^*_B(G,M))\ar[r]^{\rm{Res}}& H^*_B(G_0,M)\ar[r]^{\TR}&
 H^*_B(G,M)\,.
}$$
\end{thm}
\begin{proof} Consider $H^*_B(G,?)$ as a cohomological functor on $P(G)$. In 
\cite[Theorem 4]{Moore} Moore proved 
that this functor is
coeffaceable, that is, for $M\in P(G)$ there exist a closed
embedding $M\to J(M)$ with $J(M)\in P(G)$ satisfying
$H^*_B(G,J(M))=0$ for $*>0$. It follows that any
natural transformation $T^*: H^*_B(G,?)\to H^*_B(G,?)$
of functors on $P(G)$
is determined by $T^0: H^0_B(G,?)\to H^0_B(G,?)$. Applying that remark to the natural transformation
$T=\TR\circ\rm{Res}$ one verifies for $M\in P(G)$ that in degree $0$ 
the map $T^0:H^0(G,M)\to H^0(G,M)$ corresponds
to the map $M^G\to M^G$ given by
$m\mapsto \sum_{gH\in G/H}gm=[G:H]m$. As a result, the composite map $\TR\circ\rm{Res}$ is, as claimed, the multiplication
by $[G:H]$, in all degrees. The commutativity of the diagram
for the forgetful map is
immediate from the definition of the transfer.
\end{proof}
\section{Applications}
If $A$ is a finitely generated $\Z G$-module, with kernel
of the $G$-action on $A$ a normal 
subgroup of finite index in $H<G$, we can (as described earlier)
consider $A$ as an
object in $mP(G)$, by choosing a word metric on $A$ with
respect to a $G$-stable finite generating set. Forming
$A\otimes\R$, we can choose a $G$-invariant Euclidean metric
on that finite dimensional $\R$-vector space, yielding
an object $A\otimes\R$ in $mP(G)$ such that $A\to A\otimes\R$
is a morphism in $mP(G)$ (it maps bounded subsets to bounded subsets). 
We can also form $(A\otimes\R)/\operatorname{image}(A)$, which is topologically a finite
dimensional torus, inheriting from $A\otimes\R$ a natural $G$-invariant metric,
yielding
an exact sequence 
$$A\to A\otimes\R\to(A\otimes\R)/\operatorname{image}(A)$$
in $mP(G)$. We call this the {\em standard way}
of viewing $A$, $A\otimes\R$ and $(A\otimes\R)/\operatorname{image}(A)$
as objects in $mP(G)$.
 
\begin{lemma}\label{R-reduction}Let $G$ be a separable, locally compact topological group,
with a topology given by a complete metric and let $H<G$ a closed subgroup of finite index.
Let $A$ be a finitely generated $\Z G$-module
with $H$ acting trivially on $A$; view $A$ and $A\otimes\R$
in the standard way as an object in $mP(G)$.
If the natural map $H^*_B(G,A)\to H^*_B(G,A\otimes \R)$
maps an element $x\in H^*_B(G,A)$ to an element $\overline{x}
\in H^*_B(G,A\otimes \R)$ which lies in the image of
$H^*_{Bb}(G,A\otimes\R)\to H^*_B(G,A\otimes\R)$, then $x$
lies in the image of $H^*_{Bb}(G,A)\to H^*_B(G,A)$.
\end{lemma}
\begin{proof}
We first consider the case of an $A$ with underlying abelian
group torsion-free. In that case, tensoring with
$\Z\to\R\to S^1$ yields a short exact sequence 
$A\to A\otimes\R\to A\otimes S^1$ in $mP(G)$, which admits
a Borel bounded section $A\otimes S^1\to A\otimes\R$ of the
underlying spaces (cf. \cite{Warner}). Therefore,
we obtain a commutative diagram with exact rows
and vertical maps $\alpha$ and $\delta$ isomorphisms,
because $A\otimes S^1$ has finite diameter:
$$\xymatrix @R=2pc @C=1pc{
H^{n-1}_{Bb}(G,A\otimes S^1)\ar[r]\ar[d]^{\alpha}_\cong &
H^n_{Bb}(G,A)\ar[r]\ar[d]^{\beta}&
 H^n_{Bb}(G,A\otimes\R)\ar[r]\ar[d]^{\gamma} &
 H^n_{Bb}(G,A\otimes S^1)\ar[d]^{\delta}_\cong\\
  H^{n-1}_B(G,A\otimes S^1)\ar[r]& H^n_B(G,A)\ar[r]&
 H^n_B(G,A\otimes\R)\ar[r] & H^n_B(G,A\otimes S^1).
}$$
The result then
follows from a simple diagram chase.
If $A$ has torsion, denote by $TA<A$ the 
torsion subgroup (with the induced metric); it
is a $G$-submodule. We have a $\Z$-split
short exact sequence $TA\to A\to A/TA$
in $mP(G)$, which gives rise to a commutative
diagram with exact rows and isomorphisms
$\chi$ and $\mu$, because $TA$ is finite:
$$\xymatrix{
H^{n}_{Bb}(G,TA)\ar[r]\ar[d]^{\chi}_\cong &
H^n_{Bb}(G,A)\ar[r]\ar[d]^{\psi}&
 H^n_{Bb}(G,A/TA)\ar[r]\ar[d]^{\phi} &
 H^{n+1}_{Bb}(G,TA)\ar[d]^{\mu}_\cong\\
  H^{n}_B(G,TA)\ar[r]& H^n_B(G,A)\ar[r]&
 H^n_B(G,A/TA)\ar[r] & H^{n+1}_B(G,TA).
}$$
Having proved the claim of the lemma for $A/TA$ and using
that $A\otimes\R\cong (A/TA)\otimes\R$, a
simple diagram chase completes
the proof for the general case.
\end{proof}
\begin{lemma}\label{N-reduction}Let $G$ be a virtually connected, separable, locally compact topological group,
with a topology given by a complete metric and let $M$ in $mP(G)$ be finitely
generated as an abelian group and with discrete topology. Let $x\in H^*_B(G,M)$
and assume that for some $n>0$, $nx$ lies in the
image of $H^*_{Bb}(G,M)\to H^*_B(G,M)$. Then so does $x$.
\end{lemma}
\begin{proof} The connected component $G_0$ acts trivially
on $M$ and has finite index in $G$. From our assumption
we infer that the image $\overline{x}$ of $x$
in $H^*_B(G,M\otimes\R)$ has the property that $n\overline{x}$
lies in the image of the map $H^*_{Bb}(G,M\otimes\R)\to
H^*_B(G,H\otimes\R)$. Since that image is an $\R$-vector space,
we conclude that $\overline{x}$ lies in that image too. Using Lemma
\ref{R-reduction}, it follows that already $x$ lies in the
image of $H^*_{Bb}(G,M)\to H^*_B(G,M)$.
\end{proof}
The following result was proved in \cite{CMPS}
for connected
Lie groups and trivial coefficients $A=\Z$.
However, in applications one is often confronted
with Lie groups
with finitely many connected components (virtually connected
Lie groups), like $GL(n,\R)$ or 
$O(p,q)$. The transfer map allows a natural
generalization of \cite[Theorem 1.2]{CMPS}
to include such groups.
\begin{cor}
Let $G$ be a virtually connected Lie group. Then
the following are equivalent:
\begin{enumerate}
\item the radical $\sqrt{G}$ is linear;
\item for every finitely generated $G/G_0$-module $M$,
the forgetful map $H^*_{Bb}(G,M)\to H^*_B(G,M)$ is surjective.
\end{enumerate}
\end{cor}

\begin{proof}That (2)$\Rightarrow$(1) is the content of Corollary~\ref{3implies1}, so we show that (1)$\Rightarrow$(2). For this
we consider the commutative diagram
$$\xymatrix{
H^*_{Bb}(G,M)\ar[r]^{\rm{Res}}\ar[d]^\varepsilon&
H^*_{Bb}(G_0,M)\ar[r]^{\TR}\ar[d]^\zeta&
 H^*_{Bb}(G,M)\ar[d]^\eta\\
  H^*_B(G,M))\ar[r]^{\rm{Res}}& H^*_B(G_0,M)\ar[r]^{\TR}&
 H^*_B(G,M)\,.
}$$
Since $G_0$ is connected, the $G_0$-action on $M$ is
trivial. Because $\sqrt{G}=\sqrt{G_0}$ we conclude
from \cite[Lemma 3.5]{CMPS} that $\zeta$ is surjective.
It follows that for $x\in H^*_B(G,M)$, the element $\TR\circ\rm{Res}(x)=[G:G_0]\cdot x$ lies in
the image of $\eta$. By Lemma \ref{N-reduction}, we
conclude that $x$ lies in the image of $\eta$ too,
completing the proof.
\end{proof}
  
%%%%%%%%%%%%%%%%%%%%%%%%%%%%%%%%%%%%%%%%%%%%%%%%%%%

\end{document}